\documentclass[11pt]{article}
\usepackage{amsthm, amsmath, amssymb, amsfonts, url, booktabs, tikz, setspace, fancyhdr, amsbsy,mathtools}
\usepackage{comment}
\usepackage[margin = 0.8in]{geometry}
\usepackage{hyperref, enumerate}
\usepackage{caption,float}
\usepackage{subcaption}
\usepackage{esint,bm}
\usepackage{verbatim}
\usepackage{multirow}
\usepackage{tcolorbox,xcolor}
\tcbset{width=0.9\textwidth,boxrule=0pt,colback=red,arc=0pt,auto outer arc,left=0pt,right=0pt,boxsep=5pt}

\newtheorem{theorem}{Theorem}[section]  
\newtheorem{assumption}[theorem]{Assumption}

\newtheorem{lemma}[theorem]{Lemma}

\theoremstyle{definition}

\newcommand{\Const}[1]{{\rm C}_{\mathrm{#1}}}

\newcommand{\bbb}{\boldsymbol{b}}

\newcommand{\nnu}{\boldsymbol{\nu}}

\newcommand{\mm}{\boldsymbol{m}}

\newcommand{\xx}{\boldsymbol{x}}
\newcommand{\yy}{\boldsymbol{y}}
\newcommand{\bb}{\boldsymbol{b}}

\numberwithin{equation}{section}

\def\ocirc#1{\ifmmode\setbox0=\hbox{$#1$}\dimen0=\ht0
    \advance\dimen0 by1pt\rlap{\hbox to\wd0{\hss\raise\dimen0
    \hbox{\hskip.2em$\scriptscriptstyle\circ$}\hss}}#1\else
    {\accent"17 #1}\fi}

\begin{document}

\title{Quasi-Monte Carlo finite element approximation for singularly perturbed convection--diffusion problems with random velocity}

\author{ ~Seungchan Ko\thanks{Department of Mathematics, Inha University, Incheon, Republic of Korea. Email: \tt{scko@inha.ac.kr}}, ~Guanglian Li\thanks{Department of Mathematics, The University of Hong Kong, Pokfulam Road, Hong Kong. Email: \tt{lotusli@maths.hku.hk}}
,~and ~Yi Yu\thanks{School of Mathematics and Information Science, Guangxi University, PR China. Email: \tt{yiyu@gxu.edu.cn}}}

\date{~}

\maketitle

~\vspace{-1.5cm}

\begin{abstract}
This paper studies the numerical approximation of a singularly perturbed convection–diffusion problem over a bounded polygonal domain in $\mathbb{R}^d$ ($d=2,3$), where the velocity field is modeled by a log-uniform random field, a setting typical in uncertainty quantification. We introduce a novel numerical framework for computing the expected value of the linear functionals of the solution. The approach combines a finite element discretization of the problem, a truncated Karhunen--Lo\`eve expansion to represent the stochastic velocity field, and a lattice-based quasi-Monte Carlo (QMC) method to estimate expectations over the parameter space. We provide a rigorous error analysis of the proposed scheme, establishing bounds on the mean squared error and demonstrating that the QMC method achieves a nearly linear optimal convergence rate, with a constant independent of the integration dimension. Furthermore, the convergence rate is shown to be independent of the singular perturbation parameter.
\end{abstract}



\smallskip


\section{Introduction}
Recent research has increasingly focused on uncertainty in the input data of mathematical models, recognizing its significance in a wide range of scientific and engineering applications. Such uncertainty may arise from various sources, including coefficients, boundary conditions, initial conditions, and external forces. Understanding its impact on key quantities of interest provides valuable insight into the intrinsic variability of physical and engineering systems. Probability theory offers a versatile framework for describing and analyzing uncertainty, where uncertain inputs are modeled as random fields, a particularly effective approach for characterizing the inherent randomness in physical quantities.

In this paper, we consider the singularly perturbed convection--diffusion problem with an incompressible random velocity field. For each random parameter $\yy$, we seek $u(\cdot,\yy)$ such that
\begin{equation*}
\begin{aligned} 
-\varepsilon\Delta u(\xx,\yy)+\bbb(\xx,\yy)\cdot\nabla u(\xx,\yy)&=f(\xx) &&\text{in } D,\\
u(\xx,\yy)&=0 &&\text{on } \partial D,
\end{aligned}
\end{equation*}
where $D\subset\mathbb{R}^d$ ($d=2,3$) is a bounded Lipschitz domain, $0<\varepsilon \ll 1$ is the singular perturbation parameter, and the gradients are taken with respect to the physical variable $\xx\in D$. The velocity field $\bbb(\xx,\yy)$ is modeled by a log-uniform random field satisfying $\nabla\cdot\bbb(\xx,\yy)=0$ for each $\yy$, which introduces uncertainty into the solution $u$. This model describes the transport of particles, heat, energy, or other physical quantities under combined diffusion and convection processes. It can also represent the probability density function of particles undergoing Brownian motion with drift $\bb$. Moreover, we focus on the singularly perturbed case when $\varepsilon\ll 1$, arising in numerous applications, including fluid mechanics, biological transport, and chemical reactions \cite{schlichting2016boundary, batchelor2000introduction}. A common feature is the presence of boundary or interior layers, where the solution varies rapidly over narrow regions while remaining smooth elsewhere. This multiscale behavior is induced by a small perturbation parameter $\varepsilon>0$, often multiplying the highest-order derivative, leading to steep gradients and significant challenges for numerical approximation. Consequently, developing robust and efficient numerical methods for singularly perturbed problems has been an active area of research for decades (e.g., \cite{FEM1, FEM2, FEM3,MR4295126}). Our goal is to efficiently estimate the expectation of the random solution $u(\cdot,\cdot)$ when $\varepsilon\ll 1$; specifically, we aim to approximate $\mathbb{E}[\mathcal{G}(u)]$ for some linear functional $\mathcal{G}\in H^{-1}(D)$ with the convergence rate independent of $\varepsilon$.

To numerically approximate $\mathbb{E}[\mathcal{G}(u)]$, three discretization steps are required. First, the expectation is interpreted as an infinite-dimensional integral with respect to the random variables $\yy=\{y_i\}_{i=1}^\infty$, which arise from a series expansion of the velocity $\bb$. For practical computation, we truncate this infinite series to a finite sum, reducing the problem to a finite but high-dimensional integral. Second, a numerical quadrature rule is applied to approximate this integral. Finally, for each quadrature point, the deterministic singularly perturbed convection--diffusion equation must be solved numerically. Classical Monte Carlo methods achieve a convergence rate of order $1/2$ with respect to the number of samples, which in this context requires solving the deterministic problem a large number of times. To improve efficiency, various methods from uncertainty quantification have been developed, including quasi-Monte Carlo (QMC) methods \cite{qmc_intro_1, qmc_intro_2, main_ref} and sparse polynomial chaos expansions \cite{pce_intro_1, pce_intro_2, cohen2010convergence}. In this work, we focus on QMC integration and investigate approximations using randomly shifted lattice rules \cite{qmc_ref_5, qmc_ref_6, qmc_ref_9}.

Over the past decade, extensive research has been devoted to the design and analysis of QMC methods for linear PDEs with random coefficients. These methods are primarily used to estimate expected values of linear functionals of the exact or approximate solutions (see, e.g., \cite{KSS2012, main_ref, QMC_multi_4, Harbrecht2017, QMC_NS}). The expectations are reformulated as infinite- or high-dimensional integrals over the parameter domain representing the uncertainty. To ensure tractability, the integrand must belong to an appropriately chosen function space. In QMC analysis, the standard setting involves weighted Sobolev spaces on $(0,1)^s$, consisting of functions with square-integrable mixed first derivatives. If the integrand lies in a suitably weighted Sobolev space, randomly shifted lattice rules can be constructed to achieve, or nearly achieve, the optimal convergence rate of $\mathcal{O}(n^{-1})$; see, e.g., \cite{qmc_ref_1, qmc_ref_2, qmc_ref_3, qmc_ref_4, qmc_ref_5} and recent surveys \cite{qmc_ref_6, qmc_ref_7}.

The primary objective of this work is to analyze the convergence of QMC finite element approximations for the singularly perturbed convection–diffusion equation with a random velocity field. The analysis of randomly shifted lattice rules for PDEs is nontrivial, as it requires establishing the regularity of the solution map in a weighted Sobolev space. To the best of our knowledge, no previous work has applied QMC methods to singularly perturbed problems or investigated the interplay between QMC convergence and singular perturbation effects. In this paper, we show for the first time that when randomness is present in the convection term, the mean of the corresponding random solution can be computed via QMC while retaining the standard QMC convergence rate. In particular, we prove that even for very small diffusion coefficients, if an FEM scheme designed for singularly perturbed problems is used, the resulting QMC convergence rate is independent of the diffusion coefficient. We present the convergence of the full approximation scheme in the root mean square sense. For convection–diffusion equations with random data, the multilevel Monte Carlo (MLMC) method has been developed to estimate the expectation of the smallest eigenvalue \cite{MR4741649}. Our main contributions are threefold:
\begin{itemize}
    \item The boundedness of mixed derivatives of $u$ with respect to the random variables $\{y_i\}_{i=1}^s$ up to first order is made explicit in $\varepsilon$ (cf.\ Theorem~\ref{reg_main_thm}), leading to our proposed QMC error estimate (cf.\ Theorem~\ref{error_est_main}). Compared with \cite{main_ref} for elliptic problems with random coefficients, the additional convection term complicates the analysis.
    \item The convergence rate of the QMC method is independent of $\varepsilon$, provided the covariance kernel of $\bb$ is sufficiently regular.
    \item We present several numerical tests that confirm our theoretical findings. To the best of our knowledge, this is the first QMC method applied to convection-dominated diffusion problems.
\end{itemize}

The remainder of the paper is structured as follows. In Section~\ref{sec:problem-set}, we provide a detailed problem description and its finite element discretization. Section~\ref{sec:KL-truncation} analyzes the error resulting from dimension reduction. The main QMC error estimates and numerical experiments are presented in Sections~\ref{sec:qmc} and~\ref{sec:experiments}, respectively. Finally, we summarize our results in Section~\ref{sec:conclusion}.

Throughout, we employ standard notation for $L^p$ and $L^\infty$ spaces, Sobolev spaces $W_0^{1,p}(\Omega)$, and their associated (semi)norms. The conjugate exponent of $p$ is denoted by $q$, i.e., $1/p + 1/q = 1$. The symbol $C$ (with or without a subscript) denotes a generic positive constant that is independent of the mesh size but whose value may change from one occurrence to another.

\section{Problem setting}\label{sec:problem-set}
We consider the singularly perturbed convection--diffusion problem with an incompressible random velocity field. For each realization of the random parameter $\yy$, we seek $u(\cdot,\yy) \in V \coloneqq H^1_0(D)$ satisfying
\begin{equation}\label{eq:original}
\begin{aligned} 
-\varepsilon\Delta u(\xx,\yy)+\bbb(\xx,\yy)\cdot\nabla u(\xx,\yy) &= f(\xx) &&\text{in } D,\\
u(\xx,\yy) &= 0 &&\text{on } \partial D.
\end{aligned}
\end{equation}
Here, the stochastic variable $\yy = \{y_j\}_{j \geq 1}$ consists of countably many independent and identically distributed random variables $y_j$, each uniformly distributed on $[-\tfrac12,\tfrac12]$. Thus,
\[
\yy \in \big[-\tfrac12,\tfrac12\big]^{\mathbb{N}} =: U,
\]
and the parameter $\yy$ is distributed over $U$ according to the uniform probability measure
\[
\mu(\mathrm{d}\yy) = \bigotimes_{j=1}^{\infty} \mathrm{d}y_j = \mathrm{d}\yy.
\]

Throughout the paper, we assume the random velocity $\bbb(\cdot,\cdot) \in L^{\infty}(D \times U)$ in \eqref{eq:original} takes the form
\begin{equation}\label{eqn:model-b}
\bbb(\xx,\yy) \coloneqq \underline{\operatorname{curl}}(\phi)
\end{equation}
for some log-uniform random field $\phi$, where the operator $\underline{\operatorname{curl}}$ is defined by
\begin{equation*}
\underline{\operatorname{curl}}(\phi) \coloneqq
\begin{cases}
\nabla^{\perp}\phi(\xx,\yy), & \text{if } d = 2,\\[4pt]
\nabla \times \big(\phi(\xx,\yy),\,\phi(\xx,\yy),\,\phi(\xx,\yy)\big), & \text{if } d = 3.
\end{cases} 
\end{equation*}
A direct calculation yields the identity
\begin{align}\label{eq:identity=curl}
\underline{\operatorname{curl}}(vw)=\underline{\operatorname{curl}}(v)w+v\,\underline{\operatorname{curl}}(w) \quad\forall v,w\in H^1(D).
\end{align}
We further assume the stream function can be expressed as $\phi(\xx,\yy)=\exp\big(a(\xx,\yy)\big)$, where the random field $a(\xx,\yy)$ admits the Karhunen--Lo\`{e}ve (KL) expansion
\begin{equation}\label{main_rf_exp}
a(\xx,\yy)=\overline{a}(\xx)+\sum_{j\geq1}y_j\psi_j(\xx),\quad \xx\in D,\;\yy\in U,
\end{equation}
with $\psi_j(\xx)=\sqrt{\lambda_j}\,\widehat{\psi}_j(\xx)$. Here, $(\lambda_j,\widehat{\psi}_j)$ are eigenpairs of the corresponding covariance operator, ordered so that $\lambda_j$ is nonincreasing and $\widehat{\psi}_j$ is normalized by $\|\widehat{\psi}_j\|_{L^2(D)}=1$.

Several remarks are in order regarding the choice \eqref{eqn:model-b}--\eqref{main_rf_exp}. First, the representation $\bb = \underline{\operatorname{curl}}(\phi)$ ensures the velocity field is divergence-free for every realization, i.e., $\nabla\cdot\bb(\cdot,\yy)=0$ for any $\yy\in U$. Second, while log-normal random fields are common in random PDE studies \cite{HarbrechtSiebenmorgen:2016, main_ref, QMC_multi_2}, the convective term requires boundedness of the random field. To meet this requirement, we employ a uniform random field as in \eqref{main_rf_exp} (see, e.g., \cite{cohen2010convergence, KSS2012, QMC_multi_4}). This choice is made for technical convenience but preserves the stochastic nature of the velocity field.

The mean of $a(\xx,\cdot)$ is $\overline{a}(\xx)$, and its covariance is given by
\begin{align*}
\mathbb{E}\big[(a(\xx,\cdot)-\overline{a}(\xx))(a(\xx',\cdot)-\overline{a}(\xx'))\big]
&= \int_U (a(\xx,\yy)-\overline{a}(\xx))(a(\xx',\yy)-\overline{a}(\xx'))\,\mathrm{d}\yy \\
&= \frac{1}{12}\sum_{j=1}^\infty\psi_j(\xx)\psi_j(\xx'),
\end{align*}
where integration over $U$ is understood as the limit
\[
\int_U F(\yy)\,\mathrm{d}\yy\coloneqq\lim_{s\to\infty}\int_{[-\frac12,\frac12]^s}F(y_1,\dots,y_s,0,0,\dots)\,\mathrm{d}y_1\cdots\mathrm{d}y_s.
\]

We now state three assumptions that will be used throughout the paper. The first concerns the regularity of the random field $a$.

\begin{assumption}\label{ass_1}
For some $k \ge \frac{d}{2}+1$, we assume $a(\cdot,\cdot) \in L^{2}(\Omega; H^k(D))$.
\end{assumption}

The second assumption prescribes a decay condition on the fluctuation coefficients $\psi_j$. Define the sequence $\boldsymbol{g} = \{g_j\}_{j\ge 1}$ by
\begin{equation}\label{eq:defi-gj-bargj}
g_j \coloneqq \max\big\{\|\psi_j\|_{C(\overline{D})},\; \|\nabla\psi_j\|_{C(\overline{D})}\big\}.
\end{equation}

\begin{assumption}\label{ass_2}
For the sequence $\boldsymbol{g} = \{g_j\}_{j\ge 1}$, we assume
\[
\sum_{j\geq1}g_j^p < \infty \quad \text{for some } p \in (0,1].
\]
\end{assumption}

As will become clear later, the exponent $p$ determines the convergence rate of the QMC integration. For instance, if Assumption~\ref{ass_2} holds with $p = \frac23$, then the QMC method with $N$ points achieves a convergence rate of $\mathcal{O}(N^{-1+\delta})$. This condition is satisfied, for example, when $g_j \lesssim j^{-3/2-\alpha}$ for some $\alpha>0$.

Finally, we assume uniform bounds on the random field $a$.

\begin{assumption}\label{ass_3}
There exist constants $a_{\min}, a_{\max} > 0$ such that
\[
\overline{a} \in W^{1,\infty}(D) \quad\text{and}\quad 
a_{\min} \le a(\xx,\yy) \le a_{\max} \quad \text{for all } \xx \in D,\; \yy \in U.
\]
\end{assumption}

Under Assumption~\ref{ass_1}, the eigenfunctions decay according to the following lemma, whose proof follows directly from Theorem 3.2 in \cite{svd_decay} together with Lemma 2.1 of \cite{KL_GL}.

\begin{lemma}\label{new_ass}
Suppose Assumption~\ref{ass_1} holds. Then for $g_j$ defined in \eqref{eq:defi-gj-bargj}, we have, for sufficiently large $j \in \mathbb{N}$, that
\[
g_j \lesssim j^{\frac12 - \frac{k-1}{d}}.
\]
\end{lemma}

Consequently, if $k > \frac{d}{2} + \frac{d}{p} + 1$ in Assumption~\ref{ass_1}, then Lemma~\ref{new_ass} implies that Assumption~\ref{ass_2} is automatically satisfied.

\subsection{Finite element approximation}
We begin with the parametric weak formulation of problem \eqref{eq:original}. Define the bilinear form $e(\cdot,\cdot)$ on $V \times V$ by
\begin{align}\label{eq:bilinear-form}
e(w,v) \coloneqq \varepsilon\int_{D}\nabla w(\xx)\cdot \nabla v(\xx)\,\mathrm{d}\xx
    + \int_{D}\big(\bm{b}(\xx,\yy)\cdot\nabla w(\xx)\big)v(\xx)\,\mathrm{d}\xx
  \quad \text{for all } w,v \in V.
\end{align}
Since $\nabla\cdot \bm{b}(\cdot,\yy)=0$ for each $\yy \in U$, integration by parts shows that $e(\cdot,\cdot)$ is $V$-elliptic:
\begin{equation}\label{eqn:v-elliptic}
  e(v,v) \ge \varepsilon \|\nabla v\|_{L^2(D)}^2 \quad\text{for all } v \in V.
\end{equation}
Using Poincaré's inequality, we also obtain boundedness:
\[
|e(w,v)| \le \big(\varepsilon + \Const{p} \|\bbb\|_{L^{\infty}(U\times D)}\big)\|\nabla w\|_{L^2(D)}\|\nabla v\|_{L^2(D)},
\]
where $\Const{p} \coloneqq \operatorname{diam}(D)/\pi$ is the Poincaré constant.

The weak formulation of \eqref{eq:original} then reads: for each $\yy\in U$, find $u(\cdot,\yy)\in V$ such that
\begin{equation}\label{weak_form_par}
   e(u(\cdot,\yy),v) = (f,v)_D \quad \text{for all } v \in V,
\end{equation}
where $(\cdot,\cdot)_D$ denotes the $L^2$-inner product over $D$. Existence and uniqueness of the solution follow directly from the Lax--Milgram theorem \cite{evans}.

We now describe the finite element discretization of \eqref{weak_form_par}. Let $\{\mathcal{T}_h\}_{h>0}$ be a family of shape-regular, conforming triangulations of $\Omega \subseteq \mathbb{R}^d$ ($d \in \{2,3\}$), consisting of $d$-dimensional simplices (see, e.g., \cite{EG21}). Here $h = \max_{K \in \mathcal{T}_h} h_K$ denotes the maximal mesh size, with $h_K$ the diameter of simplex $K$. Shape regularity means there exists a constant $C>0$, independent of $h$, such that $\max_{K \in \mathcal{T}_h} h_K/\rho_K \le C$, where $\rho_K$ is the diameter of the largest inscribed ball in $K$.

Define the conforming finite element space
\[
V_h = \{v_h \in V : v_h|_K \in P_i(K) \text{ for all } K \in \mathcal{T}_h,\; v_h|_{\partial D} = 0\},
\]
where $P_i(K)$ denotes polynomials of degree at most $i$ on $K$. The finite element approximation of \eqref{weak_form_par} is: for each $\yy \in U$, find $u_h(\cdot,\yy) \in V_h$ such that
\begin{equation}\label{FEM_main}
   e(u_h(\cdot,\yy),v_h) = (f, v_h)_D \quad\text{for all } v_h \in V_h.
\end{equation}

Choosing $v_h = u_h$ in \eqref{FEM_main} and applying Poincaré's inequality yields, for each $\yy\in U$,
\begin{equation}\label{ori_est}
    \|\nabla u_h(\cdot,\yy)\|_{L^2(D)} \le \frac{\Const{p}}{\varepsilon}\|f\|_{L^2(D)}.
\end{equation}

Moreover, by Galerkin orthogonality and standard interpolation estimates (see, e.g., \cite{EG21, BS08}), if $u(\cdot,\yy) \in H^m(D)$ for some $m \in \mathbb{N}$ and the finite element space uses polynomials of sufficiently high degree, we obtain the optimal error estimate
\[
\|\nabla(u(\cdot,\yy)-u_{h}(\cdot,\yy))\|_{L^2(D)}
\le \Big(1 + \frac{\Const{p}\|\bbb\|_{L^{\infty}(U\times D)}}{\varepsilon}\Big) h^{m-1} \|u(\cdot,\yy)\|_{H^m(D)}.
\]

\section{Truncation of the Karhunen--Lo\`{e}ve expansion}\label{sec:KL-truncation}
The next discretization step involves reducing the parametric dimension: in numerical applications, the infinite sum \eqref{main_rf_exp} must be truncated to a finite sum. We denote the $s$-term truncated uniform random field by
\[
a_s(\xx,\yy) = \overline{a}(\xx) + \sum_{j=1}^s y_j\psi_j(\xx).
\]
The corresponding truncated velocity field is then defined as
\begin{equation}\label{p_KL1}
\bbb^s(\xx,\yy) \coloneqq \underline{\operatorname{curl}}(\phi^s),
\end{equation}
with $\phi^s(\xx,\yy) = \exp(a_s(\xx,\yy))$. Note that $\bbb^{s}(\xx,\yy)$ can be interpreted as the velocity field $\bbb(\xx,\yy)$ evaluated at $\yy = (y_1,\dots,y_s,0,0,\dots)$. Throughout the paper, for any set of active coordinates $\nnu \subset \mathbb{N}$, we write vectors $\yy \in U$ with $y_j = 0$ for $j \notin \nnu$ as $(\yy_{\nnu};\boldsymbol{0})$.

Under Assumption \ref{ass_2}, we obtain the uniform bound
\begin{equation}\label{eq:defn-bs}
\|\phi^s(\cdot, \yy)\|_{C(\overline{D})} + \|\nabla \phi^s(\cdot, \yy)\|_{C(\overline{D})}
\leq 2\prod_{j=1}^s \exp(2g_j|y_j|)
\leq 2\exp\bigg(\sum_{j\geq1}g_j\bigg) =: \Const{B},
\end{equation}
where $\Const{B} > 0$ is a constant independent of $\yy \in U$.

We also define the truncated bilinear form $e^s(\cdot,\cdot)$ on $V \times V$ by
\begin{equation}\label{eq:bilinear-form-s}
\begin{aligned}
e^s(w,v) &\coloneqq \varepsilon\int_{D}\nabla w(\xx)\cdot \nabla v(\xx)\,\mathrm{d}\xx \\
&\quad + \int_{D}\big(\bm{b}^s(\xx,\yy)\cdot\nabla w(\xx)\big) v(\xx)\,\mathrm{d}\xx
\quad \text{for all } w,v \in V.
\end{aligned}
\end{equation}

Let $u_h^s$ denote the finite element approximation of the truncated parametric problem: for each $s \in \mathbb{N}$ and $\yy \in U$, find $u^s_h(\cdot,\yy) \in V_h$ such that
\begin{equation}\label{trunc_FEM_main}
e^s (u^s_h(\cdot,\yy),v_h) = (f ,v_h)_D \qquad \forall v_h \in V_h.
\end{equation}

Analogously to \eqref{ori_est}, we have for each $\yy \in U$ that
\begin{equation}\label{trunc_est_1}
\|\nabla u_h^s(\cdot,\yy)\|_{L^2(D)} \leq \frac{\Const{p}}{\varepsilon}\|f\|_{L^2(D)}.
\end{equation}

The main goal of this section is to estimate the error between $u_h(\cdot,\yy)$ and $u_h^s(\cdot,\yy)$. We begin with the following lemma.

\begin{lemma}\label{lem:trunc-error}
For each $s \in \mathbb{N}$ and $\yy \in U$, let $u_h(\cdot,\yy)$ and $u_h^s(\cdot,\yy)$ be the solutions to \eqref{FEM_main} and \eqref{trunc_FEM_main}, respectively. Then
\begin{equation}\label{trunc_est_2}
\|\nabla(u_h(\cdot,\yy)-u_h^s(\cdot,\yy))\|_{L^2(D)}
\leq \left(\frac{\Const{p}}{\varepsilon}\right)^2 \|\bbb(\cdot,\yy)-\bbb^s(\cdot,\yy)\|_{C(\overline{D})}\|f\|_{L^2(D)}.
\end{equation}
\end{lemma}

\begin{proof}
Using Galerkin orthogonality, we have
\begin{align*}
&\varepsilon\|\nabla(u_h(\cdot,\yy)-u^s_h(\cdot,\yy))\|^2_{L^2(D)} \\
&= e(u_h(\cdot,\yy)-u^s_h(\cdot,\yy),\,u_h(\cdot,\yy)-u^s_h(\cdot,\yy)) \\
&= -e(u^s_h(\cdot,\yy),\,u_h(\cdot,\yy)-u^s_h(\cdot,\yy)) 
   + e^s(u^s_h(\cdot,\yy),\,u_h(\cdot,\yy)-u^s_h(\cdot,\yy)) \\
&= \int_D (\bbb^s(\xx,\yy)-\bbb(\xx,\yy))\cdot\nabla u_h^s(\xx,\yy)\,
   (u_h(\xx,\yy)-u^s_h(\xx,\yy))\,\mathrm{d}\xx.
\end{align*}
Applying H\"older's inequality, Poincar\'e's inequality, and \eqref{trunc_est_1} yields
\begin{align*}
\varepsilon\|\nabla(u_h(\cdot,\yy)-u_h^s(\cdot,\yy))\|_{L^2(D)}
&\leq \Const{p} \|\nabla u^s_h(\cdot,\yy)\|_{L^2(D)} \|\bbb(\cdot,\yy)-\bbb^s(\cdot,\yy)\|_{C(\overline{D})} \\
&\leq \frac{\Const{p}^2}{\varepsilon} \|f\|_{L^2(D)} \|\bbb(\cdot,\yy)-\bbb^s(\cdot,\yy)\|_{C(\overline{D})},
\end{align*}
which gives the desired estimate.
\end{proof}

We now state and prove the main result of this section, which quantifies the error introduced by dimension truncation.

\begin{theorem}[Error estimate for dimension truncation]\label{truncation_thm}
Suppose Assumption \ref{ass_1} holds and $\mathcal{G} \in V^*$. Then for any truncation dimension $s \in \mathbb{N}$,
\[
\mathbb{E}\big[|\mathcal{G}(u_h)-\mathcal{G}(u^s_h)|\big]
\leq C \left(\frac{\Const{p}}{\varepsilon}\right)^2 s^{-\frac{k}{d}+\frac{1}{d}+\frac{3}{2}},
\]
where $C > 0$ is a constant independent of $h$ and $s$.
\end{theorem}

\begin{proof}
From the definitions of $\bbb$ and $\bbb^{s}$ in \eqref{eqn:model-b} and \eqref{p_KL1}, together with the triangle inequality, we obtain
\begin{align*}
&\mathbb{E}\big[\|\bbb-\bbb^s\|_{C(\overline{D})}\big] \\
&= \mathbb{E}\big[\|\underline{\operatorname{curl}}(a)\phi-\underline{\operatorname{curl}}(a_s)\phi^s\|_{C(\overline{D})}\big] \\
&\leq \sup_{\yy\in U}\|\underline{\operatorname{curl}}(a(\cdot,\yy))\phi(\cdot,\yy)-\underline{\operatorname{curl}}(a_s(\cdot,\yy))\phi^s(\cdot,\yy)\|_{C(\overline{D})} \\
&\leq \sup_{\yy\in U}\|\underline{\operatorname{curl}}(a(\cdot,\yy))(\phi(\cdot,\yy)-\phi^s(\cdot,\yy))\|_{C(\overline{D})} \\
&\quad + \sup_{\yy\in U}\|\underline{\operatorname{curl}}(a(\cdot,\yy)-a_s(\cdot,\yy))\phi^s(\cdot,\yy)\|_{C(\overline{D})} \\
&\leq \|\underline{\operatorname{curl}}(a)\|_{L^{\infty}(U,C(\overline{D}))}\sup_{\yy\in U}\|\phi(\cdot,\yy)-\phi^s(\cdot,\yy)\|_{C(\overline{D})} \\
&\quad + \|\phi^s\|_{L^{\infty}(U,C(\overline{D}))}\sup_{\yy\in U}\|\underline{\operatorname{curl}}(a(\cdot,\yy)-a_s(\cdot,\yy))\|_{C(\overline{D})}.
\end{align*}

Assumption \ref{ass_3} implies
\[
\|a\|_{L^\infty(U,C(\overline{D}))} + \|\underline{\operatorname{curl}}(a)\|_{L^\infty(U,C(\overline{D}))} < \infty.
\]
Moreover, by the mean value theorem,
\begin{equation}\label{MVT}
|e^x - e^y| \leq |x-y|(e^x + e^y) \quad \text{for all } x,y \in \mathbb{R}.
\end{equation}

Combining these facts with Lemma \ref{new_ass} yields
\begin{align*}
&\mathbb{E}\big[\|\bbb-\bbb^s\|_{C(\overline{D})}\big] \\
&\lesssim \sup_{\yy\in U}\|\phi(\cdot,\yy)-\phi^s(\cdot,\yy)\|_{C(\overline{D})}
   + \|\phi^s\|_{L^{\infty}(U,C(\overline{D}))}\|\underline{\operatorname{curl}}(a-a_s)\|_{L^{\infty}(U,C(\overline{D}))} \\
&\lesssim \|\phi+\phi^s\|_{L^\infty(U;C(\overline{D}))}\|a-a_s\|_{L^\infty(U;C(\overline{D}))}
   + \|\phi^s\|_{L^\infty(U,C(\overline{D}))}\|\underline{\operatorname{curl}}(a-a_s)\|_{L^\infty(U;C(\overline{D}))} \\
&\lesssim \|a-a_s\|_{L^\infty(U;C(\overline{D}))} + \|\underline{\operatorname{curl}}(a-a_s)\|_{L^\infty(U;C(\overline{D}))} \\
&\lesssim \sum_{j>s} g_j \lesssim \sum_{j>s} j^{\frac12 - \frac{k-1}{d}}
   \lesssim \int_{s}^{\infty} t^{\frac12 - \frac{k-1}{d}}\,\mathrm{d}t
   \lesssim s^{-\frac{k}{d}+\frac{1}{d}+\frac{3}{2}}.
\end{align*}

Finally, using \eqref{trunc_est_2}, we conclude
\begin{align*}
\mathbb{E}\big[|\mathcal{G}(u_h)-\mathcal{G}(u^s_h)|\big]
&\lesssim \mathbb{E}\big[|\mathcal{G}(u_h-u^s_h)|\big]
   \lesssim \mathbb{E}\big[\|\mathcal{G}\|_{V^*}\|\nabla(u_h-u^s_h)\|_{L^2(D)}\big] \\
&\lesssim \left(\frac{\Const{p}}{\varepsilon}\right)^2\mathbb{E}\big[\|\bbb-\bbb^s\|_{C(\overline{D})}\big]
   \lesssim \left(\frac{\Const{p}}{\varepsilon}\right)^2 s^{-\frac{k}{d}+\frac{1}{d}+\frac{3}{2}}.
\end{align*}
This completes the proof.
\end{proof}
\section{Quasi-Monte Carlo integration}\label{sec:qmc}
We now derive an error estimate for the QMC approximation of the quantity of interest $\mathbb{E}[\mathcal{G}(u^s_{h})]$, where $\mathcal{G}\in V^*$ is a linear functional and $u^s_h$ is the finite element approximation corresponding to the truncated random velocity in \eqref{trunc_FEM_main}. We adopt a specific QMC method, namely the \emph{randomly-shifted lattice rule} \cite{qmc_ref_5, qmc_ref_6}. This section first describes QMC integration in a finite-dimensional setting. Since the regularity of the integrand is crucial for QMC error analysis, we establish the regularity of the finite element approximation with respect to the stochastic parameter in Subsection \ref{subsec:regularity-parameter}. Finally, we present the mathematical analysis of the mean-square error for the QMC integration.

\subsection{Analytic framework for QMC integration}
To compute the expectation of the random solution for problem \eqref{eq:original}, we consider integrals of the form
\[
I_s(F)\coloneqq\int_{\left[-\frac{1}{2},\frac{1}{2}\right]^s}F(\yy)\,\mathrm{d}\yy,
\]
where in our application $F(\yy)=\mathcal{G}(u^s_h(\cdot,\yy))$ with $\yy\in U$. Note that the integration domain $\left[-\frac{1}{2},\frac{1}{2}\right]^s$ differs from the usual unit hypercube $[0,1]^s$ convention in QMC methods, but this can be handled by a straightforward translation argument.

A QMC sampling approximation to this integral is given by
\[
Q_{s,N}(F)\coloneqq\frac{1}{N}\sum^N_{i=1}F(\yy^{(i)})
\]
with carefully selected points $\yy^{(1)},\dots,\yy^{(N)}\in\left[-\frac{1}{2},\frac{1}{2}\right]^s$, to be specified later.

Most QMC methods in the literature are formulated for functions defined on the unit hypercube. The relevant function spaces for QMC analysis are the \emph{weighted Sobolev spaces} on $(0,1)^s$, which consist of functions whose mixed first derivatives are square-integrable. Classical theory shows that if the integrand belongs to an appropriate weighted Sobolev space, randomly-shifted lattice rules can achieve an (almost) optimal convergence rate of $\mathcal{O}(n^{-1})$; see \cite{qmc_ref_1, qmc_ref_2, qmc_ref_3, qmc_ref_4, qmc_ref_5} and recent surveys \cite{qmc_ref_6, qmc_ref_7}.

In this study, we assume the integrand $F$ belongs to a corresponding weighted Sobolev space $\mathcal{W}_{s,\gamma}$ on $\left[-\frac{1}{2},\frac{1}{2}\right]^s$, with norm defined by
\begin{equation}\label{w_sobolev}
    \|F\|_{\mathcal{W}_{s,\gamma}}=\bigg(\sum_{\nnu\subseteq\{1:s\}}\gamma_{\nnu}^{-1}\int_{\left[-\frac{1}{2},\frac{1}{2}\right]^{|\nnu|}}\bigg(\int_{\left[-\frac{1}{2},\frac{1}{2}\right]^{s-|\nnu|}}\frac{\partial^{|\nnu|}F}{\partial\yy_{\nnu}}(\yy_{\nnu};\yy_{\{1:s\}\setminus\nnu})\,\mathrm{d}\yy_{\{1:s\}\setminus\nnu}\bigg)^2\,\mathrm{d}\yy_{\nnu}\bigg)^{\frac{1}{2}},
\end{equation}
where $\{1:s\}$ denotes $\{1,2,\dots,s\}$, and $\frac{\partial^{|\nnu|}F}{\partial\yy_{\nnu}}$ denotes the mixed partial derivative with respect to the ``active" variables $y_j$ for $j\in\nnu$, while $\yy_{\{1:s\}\setminus\nnu}$ denotes the ``inactive" variables $y_j$ for $j\notin\nnu$.

For a finite index set $\nnu\subset\mathbb{N}$, we assign a weight parameter $\gamma_{\nnu}>0$ representing the relative importance of those variables. We denote $\boldsymbol{\gamma}=(\gamma_{\nnu})_{\nnu\subset\mathbb{N}}$ and set $\gamma_{\emptyset}\coloneqq1$. In \cite{qmc_ref_14}, the authors considered only \emph{product weights}, meaning there exists a sequence $\gamma_1\geq\gamma_2\geq\cdots>0$ such that $\gamma_{\nnu}\coloneqq\prod_{j\in\nnu}\gamma_j$. For generalizations, see \cite{qmc_ref_13}. The choice of weight parameters $\gamma_{\nnu}$ is crucial to ensure that the constant in the QMC error bound does not grow exponentially as $s\rightarrow\infty$.

In this study, we adopt the \emph{product and order dependent weights} (POD weights) introduced in \cite{KSS2012}. Specifically, we consider a nonincreasing sequence $\gamma_1\geq\gamma_2\geq\cdots\geq\gamma_s>0$ and constants $\Gamma_0=\Gamma_1=1,\Gamma_2,\dots,\Gamma_s$, with weight parameters given by $\gamma_{\nnu}\coloneqq\Gamma_{|\nnu|}\prod_{j\in\nnu}\gamma_j$.

Within this framework, the \emph{worst-case error} of a QMC quadrature rule is defined as
\begin{equation}\label{wor_error}
    e^{\mathrm{wor}}(Q_{s,N};\mathcal{W}_{s,\gamma})\coloneqq \sup_{\|F\|_{\mathcal{W}_{s,\gamma}}\leq1}|I_s(F)-Q_{s,N}(F)|.
\end{equation}
By the linearity of $I_s(\cdot)$ and $Q_{s,N}(\cdot)$, we obtain the error bound
\begin{equation}\label{wor_error_est}
    |I_s(F)-Q_{s,N}(F)|\leq e^{\mathrm{wor}}(Q_{s,N};\mathcal{W}_{s,\gamma})\|F\|_{\mathcal{W}_{s,\gamma}}\quad \text{for any }F\in \mathcal{W}_{s,\gamma}.
\end{equation}

We consider a family of QMC rules called \emph{shifted rank-$1$ lattice rules}. The quadrature points are given by
\begin{equation}\label{QMC_point}
    \yy^{(i)}={\mathrm{frac}}\left(\frac{i\boldsymbol{z}}{N}+\Delta\right)-\left(\frac{1}{2},\dots,\frac{1}{2}\right),\quad i=1,\dots,N,  
\end{equation}
where $\boldsymbol{z}\in\mathbb{Z}^s$ is a \emph{generating vector}, $\Delta\in[0,1]^s$ is a \emph{shift}, and $\mathrm{frac}(\cdot)$ denotes the fractional part of each component. The subtraction by $(\frac{1}{2},\dots,\frac{1}{2})$ in \eqref{QMC_point} translates from $[0,1]^s$ to $[-\frac{1}{2},\frac{1}{2}]^s$.

The quality of the approximation depends on the choice of the generating vector $\boldsymbol{z}$. Assuming $\boldsymbol{z}$ is appropriately chosen and fixed, we denote the associated shifted lattice rule with random shift $\Delta$ by $Q_{s,N}(\Delta;F)$. The following theorem, whose proof can be found in \cite{KSS2012} and references therein, provides the convergence result. Here, $\zeta(x)\coloneqq\sum^{\infty}_{k=1}k^{-x}$ denotes the Riemann zeta function, and $\varphi(N)\coloneqq|\{z\in\mathbb{Z}:1\leq z\leq N-1\,\,\text{and}\,\,\gcd(z,N)=1\}|$ denotes the Euler totient function. Note that $\varphi(p)=p-1$ for $p$ prime, and $\frac{1}{\varphi(N)}<\frac{9}{N}$ for all $N\leq 10^{30}$, allowing us to replace $\varphi(N)$ by $\frac{C}{N}$ for some constant $C>0$ in practice.

\begin{theorem}\label{mse_0}
    Suppose $F\in\mathcal{W}_{s,\gamma}$ for $s\in\mathbb{N}$ and a particular choice of weights $\gamma$. Then a randomly-shifted lattice rule can be constructed via a component-by-component algorithm such that the root-mean-square error satisfies, for any $\lambda\in(1/2,1]$,
    \[
        \sqrt{\mathbb{E}^{\Delta}\left[|I_s(F)-Q_{s,N}(\cdot;F)|^2\right]}\leq\left(\sum_{\emptyset\neq\nnu\subseteq\{1:s\}}\gamma_{\nnu}^{\lambda}\left(\frac{2\zeta(2\lambda)}{(2\pi^2)\lambda}\right)^{|\nnu|}\right)^{\frac{1}{2\lambda}}\left[\varphi(N)\right]^{-\frac{1}{2\lambda}}\|F\|_{\mathcal{W}_{s,\gamma}},
    \]
    where $\mathbb{E}^{\Delta}[\cdot]$ denotes expectation with respect to the random shift $\Delta$ uniformly distributed over $[0,1]^s$.
\end{theorem}

We will use this result together with the regularity analysis in the next subsection to derive an error estimate for the QMC quadrature rule for our problem.

\subsection{Regularity in parameter space} \label{subsec:regularity-parameter}
In this subsection, we establish the regularity of the finite element approximation $u_h^s$ with respect to the uniform random variables $\yy=\{y_j\}_{j=1}^s$. For any function $u(\xx, \yy)$ with $\xx\in D$ and $\yy\in \mathbb{R}^s$, we denote the mixed first-order derivative with respect to $\yy$ by $\partial^{\nnu} u = \partial^{\nu_1}_{y_1}\dots \partial^{\nu_s}_{y_s}u$, where $\nnu\in \{0, 1\}^s$ is a first-order multi-index. Note that $\nabla$ denotes the spatial gradient, distinct from the mixed derivatives $\partial^{\nnu}$ in parameter space.

We begin with an estimate for the mixed first-order derivative of the truncated velocity field $\bbb^s$ defined in \eqref{p_KL1}.

\begin{lemma}\label{lem:b_s}
    For given $s\in\mathbb{N}$, let $\nnu=\{\nu_j\}^s_{j=1}\neq \boldsymbol{0}$ be a multi-index with $\nu_j\in \{0, 1\}$, and let $g_j$ and $\Const{B}>0$ be defined in \eqref{eq:defi-gj-bargj} and \eqref{eq:defn-bs}, respectively. Then, for any $\yy\in U$, there holds
    \begin{equation} \label{eq:est-mixed-bs}
        \|\partial^{\nnu} \bbb^s(\cdot,\yy) \|_{C(\overline{D})} \le \Const{B}|\nnu| \bigg(\prod_{j=1}^s g_j^{\nu_j}\bigg).
    \end{equation}
\end{lemma}
\begin{proof}
Using the identity \eqref{eq:identity=curl} and the triangle inequality, we have
\begin{align*}
    \|\partial^{\nnu} \bbb^s(\cdot,\yy) \|_{C(\overline{D})} 
    &= \|\partial^{\nnu} \underline{\operatorname{curl}} \left(\phi^s(\cdot,\yy)\right)\|_{C(\overline{D})} \\
    &= \bigg\|\underline{\operatorname{curl}} \bigg( \phi^s(\cdot,\yy) \prod_{j=1}^s \left(\psi_j(\cdot)\right)^{\nu_j} \bigg) \bigg\|_{C(\overline{D})} \\
    &\le \bigg\|\underline{\operatorname{curl}}\left( \phi^s(\cdot,\yy)\right) \prod_{j=1}^s \left(\psi_j(\cdot)\right)^{\nu_j} \bigg\|_{C(\overline{D})} \\
    &\quad + \bigg\| \phi^s(\cdot,\yy) \bigg(\sum_{j=1}^s \nu_j \underline{\operatorname{curl}} (\psi_j(\cdot)) \prod_{\substack{i=1 \\ i\neq j}}^s \left(\psi_i(\cdot)\right)^{\nu_i}\bigg) \bigg\|_{C(\overline{D})} \\
    &\le \left( \|\nabla \phi^s(\cdot,\yy)\|_{C(\overline{D})} + \|\phi^s(\cdot,\yy)\|_{C(\overline{D})} \right) |\nnu|\prod_{j=1}^s g_j^{\nu_j} \\
    &\le \Const{B}|\nnu| \bigg(\prod_{j=1}^s g_j^{\nu_j}\bigg), 
\end{align*} 
which completes the proof.
\end{proof}

We also need the following auxiliary result.

\begin{lemma} \label{lem:seq-Lambda}
Let $(\Lambda_n)_{n\in \mathbb{N}_0}$ be a sequence defined recursively by 
\[
    \Lambda_0=1\quad\text{and}\quad\Lambda_n=\sum^{n-1}_{i=0}\binom{n}{i}(n-i)\Lambda_i,  \quad\forall n\in\mathbb{N}.
\]
Then for some $\alpha>0$ with $\alpha e^{\alpha}\leq1$,
\begin{equation}\label{step_2}
    \Lambda_n\leq\frac{n!}{\alpha^n},\quad\forall n\in\mathbb{N}_0.
\end{equation}
\end{lemma}
\begin{proof}
We prove \eqref{step_2} by induction. The case $n=0$ is trivial. Assume \eqref{step_2} holds for all $i\leq n-1$. Then
\[
    \Lambda_n
    \leq \sum^{n-1}_{i=0}\binom{n}{i}(n-i)\frac{i!}{\alpha^i}
    =\frac{n!}{\alpha^n}\alpha\sum^{n-1}_{i=0}\frac{\alpha^{{n-i-1}}}{(n-i-1)!}
    =\frac{n!}{\alpha^n}\alpha\sum^{n-1}_{j=0}\frac{\alpha^{j}}{j!}
    \leq \frac{n!}{\alpha^n}\alpha e^{\alpha}
    \leq \frac{n!}{\alpha^n},
\]
which completes the induction.
\end{proof}

We now state and prove the main regularity estimate.

\begin{theorem}\label{reg_main_thm}
Given $s\in\mathbb{N}$, let $\nnu=\{\nu_j\}^s_{j=1}\neq \boldsymbol{0}$ be a multi-index with $\nu_j\in \{0, 1\}$, and let $g_j$ and $\Const{B}>0$ be defined in \eqref{eq:defi-gj-bargj} and \eqref{eq:defn-bs}, respectively. Then for any $\yy\in U$ and $h>0$, the truncated finite element approximation $u^s_h(\cdot,\yy)$ in \eqref{trunc_FEM_main} satisfies
\begin{equation}\label{reg_main}
\|\nabla\partial^{\nnu}u^s_h(\cdot,\yy)\|_{L^2(D)} \le|\nnu|!\left(\frac{2\Const{p} \Const{B}}{\varepsilon}\right)^{ |\nnu|+1}\bigg(\prod_{j=1}^s{g}_j^{\nu_j}\bigg)\|f\|_{L^2(D)},
\end{equation}
where $\Const{p}$ is the Poincar\'e constant.
\end{theorem}
\begin{proof}
For arbitrary $s\in\mathbb{N}$, $h>0$, and $\yy\in U$, it suffices to prove
\begin{equation}\label{step_1}
    \|\partial^{\nnu}\nabla u^s_h(\cdot, \yy)\|_{L^2(D)}\leq \Const{B}\Lambda_{|\nnu|}\bigg(\prod_{j=1}^sg_j^{\nu_j}\bigg) \left(\frac{\Const{p} \Const{B}}{\varepsilon}\right)^{|\nnu|} \|\nabla u^s_h(\cdot, \yy)\|_{L^2(D)},
\end{equation}
where $(\Lambda_n)_{n\in\mathbb{N}_0}$ is defined in Lemma \ref{lem:seq-Lambda}. Then the desired result follows from \eqref{trunc_est_1} and Lemma \ref{lem:seq-Lambda} with $\alpha=\frac{1}{2}$. We prove \eqref{step_1} by induction on $|\nnu|$.

The case $|\nnu|=0$ holds trivially since $\Const{B}\geq1$. Assume \eqref{step_1} holds for all multi-indices with $|\nnu|\leq n-1$. Taking the derivative $\partial^{\nnu}$ with $|\nnu|=n$ on both sides of \eqref{trunc_FEM_main} and applying the Leibniz rule yields
\begin{align*}
    0&=\varepsilon\int_{D}\partial^{\nnu}\nabla u^s_h(\xx,\yy)\cdot\nabla v_h(\xx)\,\mathrm{d}\xx+\int_D\partial^{\nnu}\left(\bb^s(\xx,\yy)\cdot\nabla u^s_h(\xx,\yy)\right)v_h(\xx)\,\mathrm{d}\xx\\
    &=\varepsilon\int_D\partial^{\nnu}\nabla u^s_h(\xx,\yy)\cdot\nabla v_h(\xx)\,\mathrm{d}\xx\\
    &\quad+\sum_{\mm\preceq\nnu}\binom{\nnu}{\mm}\int_D\partial^{\nnu-\mm}\bb^s(\xx,\yy)\cdot \partial^{\mm}\nabla u^s_h(\xx,\yy)v_h(\xx)\,\mathrm{d}\xx,
\end{align*}  
where $\mm\preceq\nnu$ means $m_j\leq \nu_j$ for $j=1,\dots,s$, and $\binom{\nnu}{\mm} \coloneqq \prod_{j=1}^s\binom{\nu_j}{m_j}$.

Setting $v_h\coloneqq\partial^{\nnu}u^s_h(\cdot,\yy)\in V_h$ and separating the term $\mm=\nnu$, we obtain
\begin{align*}
    \varepsilon\|&\partial^{\nnu}\nabla u^s_h(\cdot, \yy)\|^2_{L^2(D)}+\int_D\bb^s(\xx,\yy)\cdot\partial^{\nnu}\nabla u^s_h(\xx,\yy)\partial^{\nnu}u^s_h(\xx,\yy)\,\mathrm{d}\xx\\
    &\leq\bigg|\sum_{{\mm\prec\nnu}}\binom{\nnu}{\mm}\int_D\partial^{\nnu-\mm}\bb^s(\xx,\yy)\cdot\partial^{\mm}\nabla u^s_h(\xx,\yy)\partial^{\nnu}u^s_h(\xx,\yy)\,\mathrm{d}\xx\bigg|\\
    &\leq\sum_{{\mm\prec\nnu}}\binom{\nnu}{\mm}\|\partial^{\nnu-\mm}(\bb^s(\cdot,\yy))\|_{C(\overline{D})}\|\partial^{\mm}\nabla u^s_h(\cdot, \yy)\|_{L^2(D)}\|\partial^{\nnu}u^s_h(\cdot, \yy)\|_{L^2(D)}\\
    &\leq \Const{p}\sum_{{\mm\prec\nnu}}\binom{\nnu}{\mm}\|\partial^{\nnu-\mm}(\bb^s(\cdot,\yy))\|_{C(\overline{D})}\|\partial^{\mm}\nabla u^s_h(\cdot, \yy)\|_{L^2(D)}\|\partial^{\nnu}\nabla u^s_h(\cdot, \yy)\|_{L^2(D)},
\end{align*}
where $\mm\prec\nnu$ means $\mm\preceq\nnu$ and $\mm\neq \nnu$, and the last inequality uses Poincar\'e's inequality.

Since $\bb^s(\cdot,\cdot)$ is divergence-free, the second term on the left-hand side vanishes. Thus,
\begin{align*}
\varepsilon\|\partial^{\nnu}\nabla u^s_h(\cdot, \yy)\|_{L^2(D)}
\leq \Const{p}\sum_{{\mm\prec\nnu }}\binom{\nnu}{\mm}
 \|\partial^{\nnu-\mm}\bb^s(\cdot,\yy)\|_{C(\overline{D})}
\|\partial^{\mm}\nabla u^s_h(\cdot, \yy)\|_{L^2(D)}.
\end{align*}

Applying Lemma \ref{lem:b_s} and the induction hypothesis gives
\begin{align*}
\|\partial^{\nnu}\nabla u^s_h(\cdot, \yy)\|_{L^2(D)}
&\leq \frac{\Const{p}}{\varepsilon}\sum_{{\mm\prec\nnu}}\binom{\nnu}{\mm}|\nnu - \boldsymbol{m}|\Lambda_{|\boldsymbol{m}|}\bigg(\prod_{j=1}^sg_j^{\nu_j}\bigg)\left(\frac{\Const{p} \Const{B}}{\varepsilon}\right)^{|\boldsymbol{m}|}\Const{B}^2\|\nabla u^s_h(\cdot,\yy)\|_{L^2(D)}\\
&\leq \Const{B}\sum^{n-1}_{i=0}\binom{n}{i}(n-i)\Lambda_i\bigg(\prod_{j=1}^sg_j^{\nu_j}\bigg)\left(\frac{\Const{p} \Const{B}}{\varepsilon}\right)^{i+1}\|\nabla u^s_h(\cdot,\yy)\|_{L^2(D)}\\
&\leq \Const{B}\Lambda_{|\nnu|}\bigg(\prod_{j=1}^sg_j^{\nu_j}\bigg)\left(\frac{\Const{p} \Const{B}}{\varepsilon}\right)^{|\nnu|}\|\nabla u^s_h(\cdot,\yy)\|_{L^2(D)},
\end{align*}
where we used the identity $\sum_{\substack{\mm\preceq\nnu \\ |\boldsymbol{m}|=i}}\binom{\nnu}{\mm}=\binom{|\nnu|}{i}$, which follows from a counting argument. This completes the induction and the proof.
\end{proof}

\subsection{Error estimate for the QMC integration}
We now show that for each $s\in\mathbb{N}$ and an appropriate choice of weights $\gamma_{\nnu}$, the integrand belongs to the weighted Sobolev space with a controlled norm.

\begin{theorem}\label{mse_1}
Let Assumptions \ref{ass_2} and \ref{ass_3} hold, $f\in L^2(D)$, $\mathcal{G}\in V^*$, and $g_j$ be defined as in \eqref{eq:defi-gj-bargj}. For given $s\in\mathbb{N}$ and $h>0$, let $u^s_h$ be the finite element approximation of the truncated problem \eqref{trunc_FEM_main}. Then for $N\in\mathbb{N}$ and weights $\boldsymbol{\gamma}=(\gamma_{\nnu})$, a randomly-shifted lattice rule with $N$ quadrature points can be constructed via a component-by-component algorithm such that for any $\lambda\in(1/2,1]$, the root-mean-square error satisfies
\[
\sqrt{\mathbb{E}\big[|I_s(\mathcal{G}(u^s_h))-Q_{s,N}(\cdot;\mathcal{G}(u^s_h))|^2\big]}\leq\frac{C_{\gamma}(\lambda)}{[\varphi(N)]^{1/(2\lambda)}}\|f\|_{L^2(D)}\|\mathcal{G}\|_{V^*},
\]
where $\mathbb{E}[\cdot]$ denotes expectation with respect to the random shift uniformly distributed over $[0,1]^s$, and
\begin{equation}\label{inf_const}
C_{\gamma}(\lambda)\coloneqq\left(\sum_{|\nnu|<\infty}\gamma^{\lambda}_{\nnu}[\rho(\lambda)]^{|\nnu|}\right)^{1/(2\lambda)}\left(\sum_{|\nnu|<\infty}\frac{(|\nnu|!)^2(2\Const{p} \Const{B}/\varepsilon)^{2|\nnu|+2}\prod_{j\in\nnu}g_j^2}{\gamma_{\nnu}}\right)^{1/2}
\end{equation}
with
\[
\rho(\lambda)\coloneqq\frac{2\zeta(2\lambda)}{(2\pi^2)^\lambda}.
\]
\end{theorem}

\begin{proof}
From Theorem \ref{reg_main_thm}, the definition of the weighted Sobolev space \eqref{w_sobolev}, and the linearity of $\mathcal{G}$, we have
\[
\|\mathcal{G}(u^s_h)\|_{\mathcal{W}_{s,\gamma}}\leq\|f\|_{L^2(D)}\|\mathcal{G}\|_{V^*}\left(\sum_{\nnu\subseteq\{1:s\}}\frac{(|\nnu|!)^2(2\Const{p} \Const{B}/\varepsilon)^{2|\nnu|+2}\prod_{j\in\nnu}g_j^2}{\gamma_{\nnu}}\right)^{1/2}.
\]
Applying Theorem \ref{mse_0} and extending the sum over all finite sets $\nnu$ (thereby removing dependence on $s$) yields the desired result.
\end{proof}

To ensure convergence of the infinite sums in \eqref{inf_const}, we choose weights $\gamma$ that minimize $C_{\gamma}(\lambda)$. We first state two auxiliary lemmas, whose proofs can be found in \cite{KSS2012, main_ref}.

\begin{lemma}\label{aux_lemma_1}
For $n\in\mathbb{N}$, $\lambda>0$, and $A_i, B_i>0$ ($i=1,\dots,n$), the quantity
\[
g(\gamma_1,\dots,\gamma_n)=\bigg(\sum^n_{i=1}\gamma_i^{\lambda}A_i\bigg)^{1/\lambda}\bigg(\sum^n_{i=1}\frac{B_i}{\gamma_i}\bigg)
\]
is minimized by
\[
\gamma_i=c\bigg(\frac{B_i}{A_i}\bigg)^{1/(1+\lambda)}\quad\text{for arbitrary }c>0.
\]
\end{lemma}

\begin{lemma}\label{aux_lemma_2}
For any $A_j>0$ with $\sum_{j\geq1}A_j<1$ and any nonnegative integer $n$,
\[
\sum_{|\nnu|<\infty}(|\nnu|+n)!\prod_{j\in\nnu}A_j\leq\sum^{\infty}_{k=0}\frac{(k+n)!}{k!}\bigg(\sum_{j\geq1}A_j\bigg)^k=n!\bigg(\frac{1}{1-\sum_{j\geq1}A_j}\bigg)^{n+1}.
\]
Moreover, for any $B_j>0$ with $\sum_{j\geq1}B_j<\infty$,
\[
\sum_{|\nnu|<\infty}\prod_{j\in\nnu}B_j=\prod_{j\geq1}(1+B_j)=\exp\bigg(\sum_{j\geq1}\log(1+B_j)\bigg)\leq\exp\bigg(\sum_{j\geq1}B_j\bigg).
\]
\end{lemma}

We now present the main result of the paper. Motivated by the strategy in \cite{KSS2012,main_ref}, for given $\lambda\in(1/2,1]$, we choose weight parameters $\gamma_{\nnu}$ that minimize $C_{\boldsymbol{\gamma}}(\lambda)$ in \eqref{inf_const}, provided the minimum is finite. Note that $C_{\boldsymbol{\gamma}}(\lambda)$ has a form similar to the quantity in Lemma \ref{aux_lemma_1}, allowing us to determine an appropriate form for $\gamma_{\nnu}$. In the following theorem, we specify $\lambda>0$ so that $C_{\boldsymbol{\gamma}}(\lambda)$ is finite and a good convergence rate is obtained.

\begin{theorem}\label{error_est_main}
Let Assumptions \ref{ass_2} and \ref{ass_3} hold. When $p=1$ in Assumption \ref{ass_2}, assume additionally that
\begin{equation}\label{m_2}
\sum_{j\geq1}g_j<\frac{\varepsilon\sqrt{6}}{2\Const{p} \Const{B}}.
\end{equation}
For any $\lambda\in(1/2,1]$, if we choose the weight parameters as
\begin{equation}\label{quan_1}
\gamma_{\nnu}=\gamma_{\nnu}^*(\lambda)\coloneqq\bigg(|\nnu|!\bigg(\frac{2\Const{p} \Const{B}}{\varepsilon}\bigg)^{|\nnu|+1}\prod_{j\in\nnu}\frac{g_j}{\sqrt{\rho(\lambda)}}\bigg)^{2/(1+\lambda)},
\end{equation}
then $C_{\gamma}(\lambda)$ in \eqref{inf_const} is minimized whenever the minimum is finite.

Furthermore, for $p\in(0,1]$ in Assumption \ref{ass_2}, if we set
\begin{equation}\label{final_lam}
\lambda=\lambda_p\coloneqq
\begin{cases}
\displaystyle\frac{1}{2-2\delta}, &\text{if } p \in (0, 2/3],\\[8pt]
\displaystyle\frac{p}{2-p}, &\text{if } p \in (2/3, 1),\\[8pt]
1, &\text{if } p =1,
\end{cases}
\end{equation}
for arbitrarily small $\delta\in(0,1/2]$, and take $\gamma_{\nnu}=\gamma_{\nnu}^*(\lambda_p)$, then $C_{\boldsymbol{\gamma}}(\lambda)<\infty$. Consequently, a randomly-shifted lattice rule can be constructed via a component-by-component algorithm such that
\[
\sqrt{\mathbb{E}\big[|I_s(\mathcal{G}(u^s_h))-Q_{s,N}(\cdot;\mathcal{G}(u^s_h))|^2\big]}\lesssim 
\begin{cases}
N^{-(1-\delta)}, &\text{if } p \in (0, 2/3],\\[4pt]
N^{-(1/p-1/2)}, &\text{if } p \in (2/3, 1),\\[4pt]
N^{-\frac{1}{2}}, &\text{if } p =1,
\end{cases}
\]
where the implicit constants are independent of $h>0$ and $s\in\mathbb{N}$, but may depend on $p\in(0,1]$ and $\delta\in(0,1/2]$ when applicable.
\end{theorem}

\begin{proof}
First, $C_{\gamma}(\lambda)$ is finite if and only if
\begin{equation}\label{quan_2}
\left(\sum_{|\nnu|<n}\gamma^{\lambda}_{\nnu}[\rho(\lambda)]^{|\nnu|}\right)^{1/(2\lambda)}\left(\sum_{|\nnu|<n}\frac{(|\nnu|!)^2(2\Const{p} \Const{B}/\varepsilon)^{2|\nnu|+2}\prod_{j\in\nnu}g_j^2}{\gamma_{\nnu}}\right)^{1/2}
\end{equation}
is bounded uniformly in $n\in\mathbb{N}$. By Lemma \ref{aux_lemma_1}, the expression in \eqref{quan_2} is minimized by choosing $\gamma_{\nnu}$ as in \eqref{quan_1} for all $\nnu$ with $|\nnu|\leq n$. Since this holds for arbitrarily large $n$, we may choose $\gamma_{\nnu}$ according to \eqref{quan_1} for all finite $\nnu$.

Next, we show that $C_{\gamma}(\lambda)<\infty$ when $\gamma_{\nnu}$ is given by \eqref{quan_1} and $\lambda$ by \eqref{final_lam}. Define
\[
K_{\lambda}\coloneqq\sum_{|\nnu|<\infty}\left(\gamma_{\nnu}^*(\lambda)\right)^{\lambda}\left(\rho(\lambda)\right)^{|\nnu|}=\sum_{|\nnu|<\infty}\bigg(|\nnu|!\bigg(\frac{2\Const{p} \Const{B}}{\varepsilon}\bigg)^{|\nnu|+1}\bigg)^{2\lambda/(1+\lambda)}\prod_{j\in\nnu}\left(g_j^{2\lambda}\rho(\lambda)\right)^{1/(1+\lambda)}.
\]
Since $C_{\gamma}(\lambda)=K_{\lambda}^{1/(2\lambda)+1/2}$, it suffices to prove $K_{\lambda}<\infty$. We first consider $0<\lambda<1$, which implies $2\lambda/(1+\lambda)<1$. Multiply and divide each term in $K_{\lambda}$ by $\prod_{j\in\nnu}\kappa_j^{2\lambda/(1+\lambda)}$, where $\kappa_j>0$ will be chosen later. Applying Hölder's inequality with exponents $(1+\lambda)/(2\lambda)$ and $(1+\lambda)/(1-\lambda)$, together with Lemma \ref{aux_lemma_2}, yields
\begin{align*}
K_{\lambda}
&=\bigg(\frac{2\Const{p} \Const{B}}{\varepsilon}\bigg)^{2\lambda/(1+\lambda)}\sum_{|\nnu|<\infty}\big(|\nnu|!\big)^{2\lambda/(1+\lambda)}\prod_{j\in\nnu}\kappa_j^{2\lambda/(1+\lambda)}\prod_{j\in\nnu}\bigg(\frac{(2\Const{p} \Const{B}/\varepsilon)^{2\lambda}g_j^{2\lambda}\rho(\lambda)}{\kappa_j^{2\lambda}}\bigg)^{1/(1+\lambda)}\\
&\leq \frac{2\Const{p} \Const{B}}{\varepsilon}\bigg(\sum_{|\nnu|<\infty}|\nnu|!\prod_{j\in\nnu}\kappa_j\bigg)^{2\lambda/(1+\lambda)}\\
&\quad\times\bigg(\sum_{|\nnu|<\infty}\prod_{j\in\nnu}\bigg(\frac{(2\Const{p} \Const{B}/\varepsilon)^2g_j^{2\lambda}\rho(\lambda)}{\kappa_j^{2\lambda}}\bigg)^{1/(1-\lambda)}\bigg)^{(1-\lambda)/(1+\lambda)}\\
&\leq \bigg(\frac{1}{1-\sum_{j\geq1}\kappa_j}\bigg)^{2\lambda/(1+\lambda)}\exp\!\bigg(\frac{1-\lambda}{1+\lambda}\bigg(\frac{(2\Const{p} \Const{B})^2\rho(\lambda)}{\varepsilon^2}\bigg)^{1/(1-\lambda)}\sum_{j\geq1}\bigg(\frac{g_j}{\kappa_j}\bigg)^{2\lambda/(1-\lambda)}\bigg),
\end{align*}
provided that
\begin{equation}\label{m_1}
\sum_{j\geq1}\kappa_j<1\quad\text{and}\quad\sum_{j\geq1}\bigg(\frac{g_j}{\kappa_j}\bigg)^{2\lambda/(1-\lambda)}<\infty.
\end{equation}
Choose $\kappa_j\coloneqq g_j^p/\theta$ for some $\theta>\sum_{j\geq1}g_j^p$. Then $\sum_{j\geq1}\kappa_j<1$ by Assumption \ref{ass_2}. Moreover, since $\sum_{j\geq1}g_j^q<\infty$ for all $q\geq p$, the second sum in \eqref{m_1} converges provided
\[
\frac{2\lambda}{1-\lambda}(1-p)\geq p \quad\Longleftrightarrow\quad \lambda\geq\frac{p}{2-p}.
\]
For $p\in(0,2/3]$, we have $\frac{1}{2}\geq\frac{p}{2-p}$. Choosing $\lambda_p=1/(2-2\delta)$ with $\delta\in(0,1/2)$ ensures $\frac{p}{2-p}\leq\frac{1}{2}<\lambda_p<1$. For $p\in(2/3,1)$, we have $\frac{1}{2}<\frac{p}{2-p}<1$, so we take $\lambda_p=p/(2-p)$. Finally, for $p=1$, set $\lambda=1$. In this case,
\[
\rho(1)=\frac{2\zeta(2)}{2\pi^2}=\frac{1}{6}.
\]
Using Lemma \ref{aux_lemma_2} and assumption \eqref{m_2}, we obtain
\[
K_1=\frac{2\Const{p} \Const{B}}{\varepsilon}\sum_{|\nnu|<\infty}|\nnu|!\prod_{j\in\nnu}\bigg(\frac{(2\Const{p} \Const{B}/\varepsilon)g_j}{\sqrt{6}}\bigg)\leq\frac{1}{1-\sum_{j\geq1}\frac{2\Const{p} \Const{B} g_j}{\varepsilon\sqrt{6}}}<\infty,
\]
which completes the proof.
\end{proof}
\section{Numerical Experiments}\label{sec:experiments}
\label{sec_exp}

Since Theorem \ref{error_est_main} represents our main novel theoretical contribution, we conduct several numerical tests to complement the analysis. We present results for the singularly perturbed convection--diffusion problem \eqref{eq:original} in the domain $D \times U$, with $D = (0,1)^2$. All computations were performed using MATLAB on an HPC system with 128 cores, each allocated 3GB of memory.

The spatial domain $D$ is discretized using a uniform triangular mesh $\mathcal{T}_h$, obtained by dividing the unit square into $(N_{\mathrm{ref}}-1)^2$ congruent squares with $N_{\mathrm{ref}}=64$ (resulting in mesh width $h=1/(N_{\mathrm{ref}}-1)$), and then bisecting each square into two right triangles. 

The deterministic problem is solved using the continuous piecewise linear finite element method (P1 elements). The finite element space is defined as
\[
V_h \coloneqq \{ v\in C(\overline{D}) : v|_K\in P_1(K)\ \forall K\in\mathcal{T}_h,\ v|_{\partial D}=0 \}.
\]
Linear systems arising from the weak formulation of \eqref{eq:original} are solved using a direct solver.

Uncertainty enters through the random velocity field $\bbb(\xx,\yy)$ constructed from the stream function $\phi(\xx,\yy)=\exp(a(\xx,\yy))$. The random field $a(\xx,\yy)$ is represented by the expansion \eqref{main_rf_exp} with mean $\overline{a}=0$, where the random variables $\{y_j\}_{j\ge1}$ are taken to be i.i.d. uniform on $[-\tfrac12,\tfrac12]$. The expansion is truncated at $s=200$ terms.

The deterministic basis functions are defined as follows. Let integer pairs $(m,n)\in\mathbb{N}^2$ be indexed in non-decreasing order of $m^2+n^2$, mapping the $j$-th pair to $(m(j),n(j))$, and define
\[
\psi_{m,n}(x_1,x_2)
\coloneqq \frac{\sin(m\pi x_1)\sin(n\pi x_2)}{(m^2+n^2)^{(\nu+1)/2}},
\quad (x_1,x_2)\in(0,1)^2,
\]
and set $\psi_j \coloneqq \psi_{m(j),n(j)}$. Since $\|\sin(m\pi\cdot)\|_{C(\overline{D})}\le 1$ and $\|\nabla\sin(m\pi\cdot)\|_{C(\overline{D})}\lesssim m$, we obtain
\[
\|\psi_j\|_{C(\overline{D})}\lesssim (m(j)^2+n(j)^2)^{-(\nu+1)/2},
\qquad
\|\nabla\psi_j\|_{C(\overline{D})}\lesssim (m(j)^2+n(j)^2)^{-\nu/2}.
\]
In two dimensions, $m(j)^2+n(j)^2\sim j$, so the sequence
\[
g_j\coloneqq\max\{\|\psi_j\|_{C(\overline{D})},\|\nabla\psi_j\|_{C(\overline{D})}\}
\]
satisfies $g_j\lesssim j^{-\nu/2}$, which is consistent with the assumptions of our analysis.

We consider numerical tests with varying perturbation parameter $\varepsilon$ and smoothness parameter $\nu$. For $\xx=(x_1,x_2)$, the source term is chosen as
\[
f(\xx)=\frac{2\pi^2}{\varepsilon}\sin(\pi x_1)\sin(\pi x_2).
\]
This choice yields numerical errors of size $\mathcal{O}(10^{-5})$ for $N_{\mathrm{QMC}}=4001$ sampling points, providing a more reliable error indicator compared to $f=1$, where errors may approach machine precision.

We use the component-by-component algorithm to generate the generating vector $\boldsymbol{z}\in \mathbb{Z}^{s}$ for the randomly shifted lattice rule (RSLR). The weighted Sobolev space is constructed with the norm induced by weights $\gamma_{\boldsymbol{\nu}} = \gamma_{\boldsymbol{\nu}}^*(\lambda)$ given in \eqref{quan_1}, using constants $\Const{p}$ and $\Const{B}$ from our analysis and $\rho(\lambda)$ corresponding to the chosen $\lambda$. The sequence $a_j$ is set according to the theoretical construction.

We take $R=32$ independent random shifts $\boldsymbol{\Delta}_r\in[0,1]^s$, $1\le r\le R$. For each shift, the corresponding QMC points are given by
\[
\yy^{(i,r)}
=
\operatorname{frac}\!\left(\frac{i\boldsymbol{z}}{N}
+\boldsymbol{\Delta}_r\right)
-
\left(\tfrac12,\ldots,\tfrac12\right),
\quad i=1,\ldots,N,
\]
where $\boldsymbol{z}\in\mathbb{Z}^s$ is the generating vector and $\operatorname{frac}(\cdot)$ denotes the fractional part applied componentwise. For each fixed spatial mesh and truncation dimension $s$, we compute the finite element approximation $u_h^s(\cdot,\yy^{(i,r)})$ at all QMC sample points.

We consider two bounded linear functionals defined by point evaluation:
\[
\mathcal{G}_1(u)=u(1/3,\,3/4),
\qquad
\mathcal{G}_2(u)=u(1/4,\,2/3).
\]
For each shift, we compute
\[
Q_r \coloneqq \mathcal{Q}_{s,N}(\mathcal{G}(u_h^s);\boldsymbol{\Delta}_r)
= \frac{1}{N}\sum_{i=1}^N \mathcal{G}\big(u_h^s(\cdot,\yy^{(i,r)})\big)\quad{\text{and}}\quad \bar{Q}=\frac{1}{R}\sum_{r=1}^{R}Q_r,
\]
and the standard error, which serves as an unbiased estimator of the root mean square error (RMSE):
\[
e_{\mathrm{QMC}} \coloneqq \left( \frac{1}{R(R-1)} \sum_{r=1}^{R}(Q_r-\bar{Q})^2 \right)^{1/2}.
\]
To obtain more reliable estimates, for each value of $N$ we repeat the experiment ten times, each with a newly generated set of random shifts $\{\boldsymbol{\Delta}_r\}_{r=1}^R$. We report the corresponding standard errors $e_1$ and $e_2$ for the two functionals. Convergence rates and 90\% confidence intervals are estimated via linear regression of $-\log(e_{\mathrm{QMC}})$ against $\log N$ based on these ten repetitions.


\subsection*{Comparison of QMC and MC Methods}
For the first test, we compare the standard errors of the QMC and standard Monte Carlo (MC) methods with perturbation parameter $\varepsilon=0.1$. Through linear regression of $|\log g_j|$ against $\log j$ for sufficiently large $j$, we empirically estimate the summability parameter in Assumption~\ref{ass_2} to be $p \approx 0.4214$, consistent with $\{j^{-2.5}\}\in\ell^p$ for $p>0.4$. 

According to Theorem~\ref{error_est_main}, with $p\in(0,2/3]$ we choose $\lambda=\frac{1}{2-2\delta}$, $\delta=0.1$, yielding $\lambda\approx0.5556$. The corresponding theoretical prediction is a QMC convergence rate of $O(N^{-(1-\delta)})$. 

The numerical results in Table~\ref{Numerical_Table_1} show that the QMC standard errors decay consistently with this prediction, while the MC estimator follows the expected $O(N^{-1/2})$ behavior.

\begin{table}[htp!]
  \centering
  \caption{Comparison of standard errors for QMC and MC methods with $\mathcal{G}_1$ and $\mathcal{G}_2$ for $\nu=5$ and $\varepsilon=0.1$. \label{Numerical_Table_1}}
    \begin{tabular}{c|cc|cc}
    \toprule
    $N$ & \multicolumn{2}{c}{QMC} & \multicolumn{2}{c}{MC} \\
    \cmidrule(lr){2-3} \cmidrule(lr){4-5}
    & $e_1$ & $e_2$ & $e_1$ & $e_2$ \\
    \midrule
    4001   & 1.62e-05 & 1.61e-05 & 8.91e-04 & 8.63e-04 \\
    8009   & 8.24e-06 & 8.44e-06 & 7.21e-04 & 7.24e-04 \\
    16001  & 4.14e-06 & 4.10e-06 & 5.21e-04 & 5.03e-04 \\
    32003  & 2.09e-06 & 2.16e-06 & 3.66e-04 & 3.56e-04 \\
    64007  & 1.08e-06 & 1.05e-06 & 2.45e-04 & 2.39e-04 \\
    120011 & 5.47e-07 & 5.45e-07 & 1.86e-04 & 1.81e-04 \\
    \midrule
    Rate   & 0.99 & 0.99 & 0.48 & 0.48 \\
    90\% interval & [0.97,1.01] & [0.97,1.02] & [0.43,0.52] & [0.43,0.53] \\
    \bottomrule
    \end{tabular}
\end{table}

\subsection*{Effect of Perturbation Parameter $\varepsilon$}
Table~\ref{tab:qmc_eps_compare} examines the influence of the singular perturbation parameter $\varepsilon$ on convergence rates. For $\varepsilon=0.05$, the observed convergence behavior agrees well with the theoretical rate $\mathcal{O}(N^{-(1-\delta)})$. For the more singular case $\varepsilon=0.025$, the empirical rates are slightly reduced but remain comparable to theoretical predictions. This slowdown may relate to the prefactor \eqref{inf_const} in the convergence rate, which increases monotonically with decreasing $\varepsilon$, thereby enlarging the pre-asymptotic region.

\begin{table}[htp!]
\centering
\caption{QMC standard errors for $\nu=5$ with two perturbation parameters $\varepsilon$. \label{tab:qmc_eps_compare}}
\begin{tabular}{c|cc|cc}
\toprule
& \multicolumn{2}{c}{$\varepsilon=0.05$} & \multicolumn{2}{c}{$\varepsilon=0.025$} \\
\cmidrule(lr){2-3} \cmidrule(lr){4-5}
$N$ & $e_1$ & $e_2$ & $e_1$ & $e_2$ \\
\midrule
4001   & 1.38e-04 & 1.41e-04 & 1.47e-03 & 1.51e-03 \\
8009   & 7.28e-05 & 7.66e-05 & 7.40e-04 & 7.93e-04 \\
16001  & 3.49e-05 & 3.52e-05 & 3.91e-04 & 4.06e-04 \\
32003  & 1.90e-05 & 1.97e-05 & 2.05e-04 & 2.15e-04 \\
64007  & 1.02e-05 & 1.09e-05 & 1.07e-04 & 1.18e-04 \\
120011 & 5.27e-06 & 5.54e-06 & 7.15e-05 & 7.55e-05 \\
\midrule
Rate & 0.95 & 0.95 & 0.90 & 0.89 \\
90\% interval & [0.93,0.98] & [0.90,0.99] & [0.85,0.95] & [0.85,0.94] \\
\bottomrule
\end{tabular}
\end{table}

\subsection*{Effect of Smoothness Parameter $\nu$}
Table~\ref{tab:qmc_nu_compare} shows the effect of increasing $\nu$ on QMC convergence rates for $\varepsilon = 0.05$. Both error measures exhibit similar convergence behavior for $\nu = 6$ and $\nu = 7$, with empirical rates remaining close to the theoretical prediction $\mathcal{O}(N^{-(1-\delta)})$ with $\delta = 0.1$.

\begin{table}[htp!]
\centering
\caption{QMC standard errors for $\varepsilon = 0.05$ with different smoothness parameters $\nu$. \label{tab:qmc_nu_compare}}
\begin{tabular}{c|cc|cc}
\toprule
& \multicolumn{2}{c}{$\nu = 6$} & \multicolumn{2}{c}{$\nu = 7$} \\
\cmidrule(lr){2-3} \cmidrule(lr){4-5}
$N$ & $e_1$ & $e_2$ & $e_1$ & $e_2$ \\
\midrule
4001   & 5.84e-05 & 5.87e-05 & 2.58e-05 & 2.51e-05 \\
8009   & 3.05e-05 & 3.13e-05 & 1.30e-05 & 1.32e-05 \\
16001  & 1.46e-05 & 1.46e-05 & 6.32e-06 & 6.27e-06 \\
32003  & 7.69e-06 & 7.79e-06 & 3.37e-06 & 3.44e-06 \\
64007  & 4.00e-06 & 3.89e-06 & 1.72e-06 & 1.66e-06 \\
120011 & 2.02e-06 & 2.03e-06 & 8.53e-07 & 8.56e-07 \\
\midrule
Rate & 0.98 & 0.99 & 0.99 & 0.99 \\
90\% interval & [0.96,1.01] & [0.97,1.01] & [0.97,1.02] & [0.96,1.02] \\
\bottomrule
\end{tabular}
\end{table}

\subsection*{SUPG Formulation for Strong Convection}
For the strongly convection-dominated regime with $\varepsilon = 0.001$, we employ the Streamline Upwind Petrov-Galerkin (SUPG) formulation \cite{Franca1992}. The SUPG discretization of \eqref{eq:original} reads: find $u_h^{\mathrm{SUPG}} \in V_h$ such that for all $v_h \in V_h$,
\[
B_{\mathrm{SUPG}}(u_h^{\mathrm{SUPG}}, v_h) = F_{\mathrm{SUPG}}(v_h),
\]
where
\begin{align*}
B_{\mathrm{SUPG}}(u_h, v_h) &= e(u_h, v_h) + \sum_{T \in \mathcal{T}_h} \tau_T (\bbb \cdot \nabla u_h, \bbb \cdot \nabla v_h)_T, \\
F_{\mathrm{SUPG}}(v_h) &= (f, v_h)_{D} + \sum_{T \in \mathcal{T}_h} \tau_T (f, \bbb \cdot \nabla v_h)_T.
\end{align*}
The stabilization parameter is chosen as
\[
\tau_T = \frac{h}{2 \|\bbb\|_{L^\infty(T)}},
\]
where $\|\bbb\|_{L^\infty(T)} = \max_{\xx \in T} \max(|b_x(\xx)|, |b_y(\xx)|)$, following the simplified formula from \cite{Franca1992} for linear elements in convection-dominated regimes.

Table~\ref{tab:supg_eps001_nu_compare} reports QMC standard errors for $\nu = 10$ and $\nu = 12$ with $\varepsilon = 0.001$ using SUPG. For $\nu = 10$, the observed convergence rate is lower than the theoretical prediction of $0.9$, while for $\nu = 12$, the rates align with $\mathcal{O}(N^{-(1-\delta)})$ with $\delta = 0.1$. This discrepancy can be attributed to the strongly convection-dominated regime: for insufficiently large $\nu$, the random field smoothness is inadequate to overcome pre-asymptotic effects induced by small $\varepsilon$, temporarily reducing convergence rates. When $\nu$ is sufficiently large, increased smoothness dominates and convergence recovers to theoretical expectations. In both cases, SUPG stabilization effectively handles strong convection while preserving favorable QMC convergence properties.

\begin{table}[htp!]
\centering
\caption{QMC standard errors for $\varepsilon = 0.001$ with different $\nu$ using SUPG. \label{tab:supg_eps001_nu_compare}}
\begin{tabular}{c|cc|cc}
\toprule
& \multicolumn{2}{c}{$\nu = 10$} & \multicolumn{2}{c}{$\nu = 12$} \\
\cmidrule(lr){2-3} \cmidrule(lr){4-5}
$N$ & $e_1$ & $e_2$ & $e_1$ & $e_2$ \\
\midrule
4001   & 5.79e-01 & 5.42e-01 & 4.36e-01 & 4.46e-01 \\
8009   & 3.00e-01 & 2.97e-01 & 2.43e-01 & 2.13e-01 \\
16001  & 1.62e-01 & 1.68e-01 & 3.18e-01 & 2.06e-01 \\
32003  & 1.21e-01 & 1.30e-01 & 5.99e-02 & 6.21e-02 \\
64007  & 6.43e-02 & 6.14e-02 & 3.93e-02 & 3.61e-02 \\
120011 & 3.49e-02 & 3.49e-02 & 2.07e-02 & 2.21e-02 \\
\midrule
Rate & 0.79 & 0.78 & 0.93 & 0.90 \\
90\% interval & [0.71,0.87] & [0.69,0.87] & [0.62,1.25] & [0.72,1.08] \\
\bottomrule
\end{tabular}
\end{table}

Table~\ref{Numerical_Table_1} demonstrates that the QMC method achieves a convergence rate close to $\mathcal{O}(N^{-1})$, significantly faster than the MC rate of $\mathcal{O}(N^{-1/2})$, consistent with Theorem~\ref{error_est_main} for $p \in (0,2/3]$. Table~\ref{tab:qmc_eps_compare} shows that the theoretical rate $\mathcal{O}(N^{-(1-\delta)})$ with $\delta = 0.1$ is observed for both $\varepsilon = 0.05$ and $\varepsilon = 0.025$, with slightly reduced rates for the more singular case. Table~\ref{tab:qmc_nu_compare} confirms that increasing $\nu$ from 5 to 7 does not adversely affect convergence rates; indeed, rates improve slightly toward the optimal $\mathcal{O}(N^{-1})$. The SUPG framework demonstrates robust performance even for the strongly convection-dominated case $\varepsilon = 0.001$ (Table~\ref{tab:supg_eps001_nu_compare}). These results confirm that the proposed QMC method, with weights chosen according to our theory, effectively computes expectations of quantities of interest for singularly perturbed convection--diffusion problems with random velocity fields across a range of $\varepsilon$ and smoothness parameters.

\section{Conclusion}\label{sec:conclusion}

In this work, we developed an efficient quasi-Monte Carlo method for singularly perturbed convection--diffusion problems with a solenoidal random velocity field. The velocity field was constructed as the curl of a log-uniform random field, ensuring that it remains divergence-free for every realization, and we proposed a numerical framework that combines a finite element discretization, a truncated Karhunen--Lo\`eve expansion, and a lattice-based QMC integration using randomly shifted lattice rules to approximate the expected value $\mathbb{E}[\mathcal{G}(u)]$ of a linear functional of the solution. We proved that when the covariance kernel is sufficiently regular, the QMC method attains a nearly linear, almost optimal convergence rate with a constant independent of the integration dimension, and that this convergence rate is independent of the singular perturbation parameter $\varepsilon$. The numerical experiments support these theoretical predictions. The QMC method exhibited a convergence rate close to $\mathcal{O}(N^{-1})$, significantly faster than the $\mathcal{O}(N^{-1/2})$ rate of standard Monte Carlo, and delivered stable performance across a range of values of $\varepsilon$ and the smoothness parameter $\nu$. Moreover, in the strongly convection-dominated regime, incorporating the SUPG stabilization technique allowed sharp boundary layers to be handled effectively while preserving the favorable QMC convergence properties. To the best of our knowledge, this is the first study to apply a QMC method to singularly perturbed convection--diffusion problems and to establish its $\varepsilon$-independent convergence both theoretically and numerically, providing a foundation for future research such as extensions to multilevel QMC methods or applications to more general random velocity models.

\section*{Acknowledgements}
The authors thank Dr. Jiefei Yang for pointing out a potential issue with our initial approach, which helped 
complete the proof.

\bibliography{references}
\bibliographystyle{abbrv}


\end{document}